\documentclass[a4paper,10pt]{amsart}

\usepackage[a4paper,width=150mm,top=25mm,bottom=25mm]{geometry}

\usepackage[utf8]{inputenc}
\usepackage{graphicx}
\graphicspath{ {images/} }
\usepackage{xcolor}
\usepackage{latexsym}
\usepackage{lipsum}
\usepackage{nicefrac}
\usepackage[english]{babel}
\usepackage[T1]{fontenc}
\usepackage{amsmath}
\usepackage{amssymb}
\usepackage{mathtools}
\usepackage{fge}
\usepackage{amsthm}
\usepackage{stackengine}
\stackMath

\usepackage{braket}
\usepackage{tikz-cd}
\usetikzlibrary{ shapes }
\usepackage[bookmarksopen=true]{hyperref}
\definecolor{darkblue}{RGB}{0,0,160}
\definecolor{darkgreen}{RGB}{0,80,0}
\hypersetup{
colorlinks,%
citecolor=darkgray,%
filecolor=darkgray,%
linkcolor=darkgray,%
urlcolor=darkgray,%
}
\usepackage[noabbrev]{cleveref}
\usepackage[font={small}]{caption}
\usepackage{verbatim}
\usepackage{mathrsfs}
\usepackage[cal=boondox,scr=boondoxo]{mathalfa}
\usepackage{enumitem}
\setlist[enumerate]{itemsep=5pt, labelindent=0pt, leftmargin=*, topsep=5pt, itemindent=\parindent}
\usepackage[utf8]{inputenc}
\usepackage[numbers,sort]{natbib}

\usepackage{nicematrix}
\usetikzlibrary{patterns}
\newtheorem{numerable}{Numerable}
\newtheorem{lemma}[numerable]{Lemma}
\newtheorem{proposition}[numerable]{Proposition}
\newtheorem{theorem}[numerable]{Theorem}
\newtheorem{corollary}[numerable]{Corollary}

\numberwithin{numerable}{section}
\theoremstyle{definition}
\newtheorem{definition}[numerable]{Definition}
\newtheorem{example}[numerable]{Example}

\newtheorem{convention}[numerable]{Convention}
\newtheorem{remark}[numerable]{Remark}

\numberwithin{equation}{section}
\definecolor{details}{RGB}{0,0,255}
\definecolor{task}{RGB}{0,191,0}
\definecolor{sketch}{RGB}{255,0,0}

\newcommand{\D}{\ensuremath{\Delta}}
\newcommand{\Der}{\operatorname{Der}}
\newcommand{\Hom}{\operatorname{Hom}}
\newcommand{\K}{\ensuremath{\mathbb{K}}}

\newcommand{\Ndelred}[2]{\widetilde{N}_{#1}(#2)}

\newcommand{\Ndel}[2]{N_{#1}(#2)}
\newcommand{\N}{\mathbb{N}}
\newcommand{\M}{\ensuremath{\mathcal{M}}}

\newcommand{\Z}{\mathbb{Z}}
\newcommand{\bfa}{\ensuremath{\mathbf{a}}}
\newcommand{\bfb}{\ensuremath{\mathbf{b}}}
\newcommand{\bfc}{\ensuremath{\mathbf{c}}}

\newcommand{\bff}{\ensuremath{\mathbf{f}}}

\newcommand{\card}[1]{\#{#1}}

\newcommand{\cone}{\ensuremath{\operatorname{cone}}}

\newcommand{\deff}[1]{{\textbf{#1}}}
\newcommand{\definedas}{\coloneqq}
\newcommand{\defi}[1]{\textbf{#1}}

\newcommand{\firstcotangentcohomology}[2]{T^1_{#1}(#2)}

\newcommand{\incgraphred}[2]{G^\circ_{#1}(#2)}
\newcommand{\incgraph}[2]{G_{#1}(#2)}
\newcommand{\link}[2][M]{\ensuremath{\operatorname{{link}}_{#1}#2}}
\newcommand{\lk}{\ensuremath{\operatorname{link}}}
\newcommand{\matroidrel}[3]{\mathcal{#1}_{#2}(#3)}
\newcommand{\matroid}[2]{\mathcal{#1}_{#2}}
\newcommand{\opensimplex}[1]{\ensuremath{\braket{#1}}}
\newcommand{\pathcomponents}[1]{\pi_0(#1)}

\newcommand{\rank}{\ensuremath{\operatorname{rk}}}
\newcommand{\rk}[2]{\ensuremath{\operatorname{rk}_{#1}(#2)}}
\newcommand{\sdif}{\ensuremath{\,\scalebox{.8}{$\fgebackslash$}\,}}

\newcommand{\sodass}{~\,:\,~}
\newcommand{\str}{\ensuremath{\operatorname{star}}}

\newcommand{\supp}{\operatorname{supp}}
\newcommand{\too}{\ensuremath{\longrightarrow}}
\newcommand{\uniform}[2]{U_{#1}^{#2}}

\renewcommand{\min}{\operatorname{min}}

\begin{document}

\title{
The first Cotangent Cohomology Module for Matroids
}
\author{William Bitsch}
\email{bitschwilliam@gmail.com}
\address{Institute of Mathematics, Freie Universit\"at Berlin, Germany}
\author{Alexandru Constantinescu}
\email{aconstant@zedat.fu-berlin.de}
\address{Institute of Mathematics, Freie Universit\"at Berlin, Germany}

\begin{abstract}
  We find a combinatorial formula which computes the first cotangent cohomology module of Stanley-Reisner rings associated to matroids. 
  For arbitrary simplicial complexes we provide upper bounds for the dimensions of the multigraded components of $T^1$. For specific degrees we prove that these bounds are reached if and only if the simplicial complex is a matroid, obtaining thus a new characterization for matroids.
  Furthermore, the graded first cotangent cohomology turns out to be a complete invariant for nondiscrete matroids.   
\end{abstract}
\maketitle
\section{Introduction}
The cotangent cohomology modules $T^i$  (or André-Quillen cohomology modules) of a commutative ring are obtained from the derived functor of the derivation functor \cite{And74}. In  cohomological degrees one and two these had been previously introduced by Lichtenbaum and Schlessinger \cite{LS67}. The interest in these two degrees comes from deformation theory. The first module   parametrizes first order deformations up to isomorphism; i.e. deformations with parameter space $\K[\varepsilon]/(\varepsilon^2)$. The second module contains all obstructions to lifting such deformations to larger parameter spaces, but may contain more than that.

In this paper  we view matroids as abstract simplicial complexes whose  maximal faces  satisfy the basis-exchange axiom.
Our main contribution is a complete characterization of  $T^1$  for the Stanley-Reisner rings associated to matroids. We also show one can read from $T^1$ whether a simplicial complex is a matroid or not.  We prove that only for matroids  one can  recover from $T^1$ the combinatorial structure\footnote{ Unless the matroid consists \emph{only} of loops and coloops. These are precisely the matroids with $T^1=0$}.
The main tool that we use is the  description  of the multigraded components of these modules for arbitrary Stanley-Reisner rings   in terms of the relative cohomology of certain topological spaces   given by Altmann and Christophersen in \cite{altmann2000stanleyreisner}.

Other  algebraic properties of simplicial complexes completely characterize matroids. For instance, the symbolic powers of a radical monomial ideal are Cohen-Macaulay precisely when the simplicial complex is a matroid  \cite{Var11,MT11}. Besides adding a new algebraic characterization for matroids, our   motivation comes from the possibility to construct flat families whose special fibre is the Stanley-Reisner ring of the matroid. Constructing explicit deformations
of the projective scheme associated to matroidal Stanley-Reisner rings may prove useful for understanding numerical invariants of matroids. This is because, thanks to upper semicontinuity, homological invariants are preserved; in particular the $h$-vector is constant on the fibers.

In Section 2 we recall  terminology, we fix notation, and we briefly present the main tools from \cite{altmann2000stanleyreisner} that we will use.
In Section 3 we  refine some techniques for computing the cotangent cohomology which we apply to arbitrary simplicial complexes.
The main result in this section is an upper bound on the vector space dimension of the  graded components of $T^1$ (Porpositions~\ref{prop:BoundOfDegreewiseDimensionOfFirstCCMOne}~and~\ref{prop:dependentCotangentCohomology}).
In Section 4 we calculate the first cotangent cohomology of matroidal Stanley-Reisner rings  and prove that a simplicial complex is a matroid if and only if the dimensions of the components of $T^1$ are obtained from our formula (Theorem~\ref{thm:FirstCotangentCohomologyMatroid}). It turns out that, to determine if a simplicial complex on $[n]$ is a matroid, it is enough to know the dimensions of the $\Z^n-$graded components $T^1_{-e_{i}}$ for $i=1,\dots, n$ (Corollary~\ref{FirstCotangentFormulaDeterminesMatroid}).
In Section 5 we show how one can recover the independent sets of a  matroid from the graded cotangent cohomology module of the associated  Stanley-Reisner ring (Theorem~\ref{thm:ReconstructMatroidsFromCCM}). We start by characterizing the algebraically rigid matroids, namely those with $T^1=0$.
These turn out to be precisely the discrete matroids: those which are the join of a simplex with some loops (Corollary~\ref{cor:trivialcotangentcohom.}). This is in contrast with arbitrary simplicial complexes, for which a complete characterization of algebraic rigidity is still missing. Significant partial results on this topic were obtained by the authors of \cite{altmann2016rigidity}.

\section{Preliminaries and notation}\label{sec:Abstract simplicial complexes and matroids}

\subsection{Combinatorics}
Matroids were first defined in the 1930s to abstract the combinatorics of linear independence.
They did so with remarkable success. 
Given a finite set of vectors $E$, one is interested in the combinatorial structure of the subsets of vectors that are linearly independent. Let $\card{}$ denote the cardinality of a set. The properties which are abstracted are:
\begin{enumerate}[label={({\bf I\arabic*})}]
\item\label{item:indep1}\textit{If $I$ is an independent set and $J\subseteq I$, then $J$ is also an independent set.}
\item\label{item:indep2}\textit{If $I, J$ are independent and $\card{J}<\card{I}$, there exists $v\in I\sdif J$ such that $J\cup\set{v}$ is independent.}
\end{enumerate}
A collection $\Delta\subseteq 2^E$ of subsets of a finite set $E$ satisfying only condition \ref{item:indep1} is called an \deff{abstract simplicial complex}\footnote{ We will  usually drop the word \emph{abstract} and just use \deff{simplicial complex}.} on the vertex set $E$. Unless otherwise stated, we  assume that for some positive natural number $n$ we have  $E=[n]=\set{1,\dots,n}$.  The subsets in $\Delta$ will be called \deff{faces}. The faces which are maximal under inclusion will be called \deff{facets}. A subset of $C\subseteq [n]$ is a \deff{nonface} of \D\ if $C\notin \D$; if all proper subsets of $C$ are in \D, then $C$ is called a minimal nonface.
In accordance with the upcoming definitions for matroids, we will denote for all simplicial complexes:
\begin{eqnarray*}
  \matroid{C}{\D}&=&\Set{C\subseteq [n]\sodass C\text{~is a minimal nonface of~}\D},\\
  \matroid{B}{\D}&=&\Set{B\subseteq [n]\sodass B\text{~is a facet  of~}\D}.
\end{eqnarray*}
A \deff{matroid}  is a nonempty\footnote{ This means that $\Delta=\emptyset$. If a simplicial complex satisfies $\Delta\neq \emptyset$, then by \ref{item:indep1} we have $\emptyset\in\Delta$. For matroids this last condition is usually included as an axiom.}
 simplicial complex whose faces satisfy \ref{item:indep2}. We will call the faces of a matroid  \deff{independent sets} and the facets of a matroid \deff{bases}. The minimal nonfaces of a matroid are called \deff{circuits}. Matroids have equivalent characterizations in terms of their bases \cite[Section\,1.2]{OXLEY} or of their circuits.
We briefly recall the latter, as it will be used later. A nonempty simplicial complex $\Delta$ is a matroid if and only if its minimal nonfaces satisfy the \deff{strong circuit elimination axiom} \cite[Proposition~1.4.12]{OXLEY}: 
\begin{enumerate}[label={({\bf C3'})},leftmargin = 3em, labelwidth=3em, itemindent = 0em, labelindent = 0em, listparindent = 3em]
\item\label{item:circEl1} \textit{If $C$ and $C'$ are distinct minimal nonfaces, $i\in C\cap C'$ and $v\in C\sdif C'$ then there exists a minimal nonface $C''$ with
  $ v\in C''\subseteq \left(C\cup C'\right)\sdif\Set{i}$.}
\end{enumerate}
We will use \D\ to denote simplicial complexes which are not necessarily matroids and reserve the notation \M\ for matroids.

For every simplicial complex $\Delta$ on $[n]$ we define the \deff{rank function} $\rank_\Delta:2^{[n]}\too \N$ by
\[
  \rk{\Delta}{A} = \max\Set{ \card{F} \sodass F \subseteq A~\text{and~} F \in \Delta },
\]
Matroids can be viewed as simplicial complexes whose rank function satisfies the semimodular inequality \cite[Chapter\,1.3]{OXLEY}. The rank of the simplicial complex  will be the maximum value of the rank function. In particular $\rank\Delta =\max\Set{\card{F}\sodass F\in\D}$.

A \deff{loop} is an element $v\in[n]$ with $\set{v}\notin \D$; equivalently, $v$ is not contained in any face of \D. A \deff{coloop} is a vertex $v\in [n]$ which is contained in every facet; alternatively, a coloop is not contained in any minimal nonface.
Let $W$ be a subset of $[n]$. The \deff{restriction} of \D\ to  $W$ is the simplicial complex on $W$ given by
\[
  \D|_W = \Set{F \in \D\sodass F \subseteq W}.
\]
The \deff{deletion} of $W$ is the restriction to the complement of $W$ in $[n]$:
\[
  \D\sdif W =\D|_{[n]\sdif W}.
\]
Given another simplicial complex $\Gamma$, the  \defi{join} of $\Delta$ and $\Gamma$ is
\[
  \Delta\ast\Gamma = \Set{F\sqcup G\sodass F\in \Delta \text{~and~}G\in\Gamma},
\]
where $\sqcup$ stands for the disjoint union. The \defi{link}\footnote{ This is a particular case of contraction, which can defined for every subset $[n]$, not just for faces. For our purposes here, the link of a face will suffice.} of a face $F\in \Delta$ is defined as 
\[
    \link[\Delta]{F} =  \Set{A \in \D\sodass A \cap F = \emptyset\text{~and~} A \cup F \in \D}.
\]
For every finite set $F$, the abstract simplex on $F$ is $2^F=\Set{A\subseteq F}$. The \deff{star} of a face $F\in\Delta$ is
\[
  \str_\Delta F =  2^F \ast \lk_\Delta F  = \Set{G \in\Delta \sodass F\cup G\in\Delta}.
\]

\subsection{Algebra}
Let \K\ be an arbitrary field and $S=\K[x_1,\dots,x_n]$ the polynomial ring in $n\in\N_{>0}$ variables with coefficients in \K. To every simplicial complex \D\ on $[n]$ we associate a radical monomial ideal of $S$ called its \deff{Stanley-Reisner ideal}:
\[
  \textstyle
  I_\D =\left\langle~  \prod_{i \in F} x_i \sodass F \in 2^{[n]}\sdif \D \right\rangle \subseteq S.
\]
This gives a bijection between simplicial complexes on $[n]$ and radical monomial ideals of $S$. The quotient ring $\K[\D]=S/I_\D$ is called the \deff{Stanley-Reisner ring} of \D\ over the field \K. If $\Delta$ is a matroid, we will call the associated Stanley-Reisner ring or ideal \defi{matroidal}. \\[1ex]
\indent We will now introduce the first cotangent cohomology module for Stanley-Reisner rings in the ad hoc way of \cite{altmann2000stanleyreisner}. For the general homological theory we refer to the books of André \cite{And74} and of Loday \cite{Lod13}, and for the connection to deformation theory we refer to Hartshorne's book \cite{HAr09} and Sernesi's book \cite{Ser07}.
While some  algebraic structures related to Stanley-Reisner rings we are about to introduce depend on the choice of field and its characteristic, the \K-vector space  dimensions of the cotangent cohomology modules depend only on the combinatorics of the complex \cite[Corollary\,1.4]{altmann2016rigidity}. As we are only interested in these dimensions,  we will for simplicity not mention ``over \K'' in the following definitions.

 For the polynomial ring $S=\K[x_1,\dots,x_n]$ we denote by
\[
  \Der_\K(S,S)=\Set{\partial\in\Hom_\K(S,S) \sodass \partial(f g) = f \partial(g) + \partial(f) g,~~\forall~f,g \in S},
\]
the $S-$module of the $\K-$linear derivations.
For any ideal $I\subseteq S$ the \deff{first cotangent cohomology module} $T^1(S/I)$ is the cokernel of the natural map $\Der_\K(S,S)\too\Hom_S(I,S/I)$.
As $I\subseteq S$ is a monomial ideal, the Stanley-Reisner ring, its resolution, and all the modules defined above are $\Z^n-$graded.
For $\bfc\in\Z^n$  we denote the $\Z^n-$graded components of the first cotangent cohomology module by
\[
  T^1_\bfc(S/I).
\]
Once the field \K\ is fixed, we will denote simply by $T^{1}(\D)$ the first cotangent cohomology of $S/I_\D$.  We call a simplicial complex \D\ \defi{algebraically rigid} if $T^1(\D)=0$.  The complex  \D\ is called \defi{$\emptyset$-rigid} if $T^1_\bfc(\D) =0$ for all $\bfc\in \Z_{\leqslant 0}^n$. The support of a vector $\bfa\in\N^n$ is defined as the set $\supp\bfa = \Set{i\in[n]\sodass a_i\neq 0}\subseteq [n]$. We will write every vector $\bfc\in\Z^n$ as
\[
  \bfc=\bfa-\bfb \quad\text{with ~~$\bfa, \bfb\in\N^n$ ~~and~~ $\supp \bfa \cap \supp \bfb = \emptyset$.}
\]
In this notation, $\emptyset$-rigid means $T^1_{-\bfb}(\D)=0$ for all $\bfb\in\N^{n}$. We paraphrase the following result. 
\begin{lemma}[{\cite[Lemma 2]{altmann2000stanleyreisner}}]
  \label{lem:L2AC}
  The module $T^1_{\bfa-\bfb}$ vanishes unless $0\neq\bfb\in\{0,1\}^n$, $\supp\bfa\in\D$ and $\supp \bfb \subseteq [\lk_\D \supp \bfa]$\footnote{ Where $[\D]:=\Set{v\in[n]\sodass v\in\D}$ denotes the set of vertices appearing in $\D$.}.
With these conditions  fulfilled, $T^1_{\bfa-\bfb}$  depends only on $\supp\bfa$ and \bfb.
\end{lemma}
 \noindent Furthermore, in \cite[Proposition~11]{altmann2000stanleyreisner} it is shown that  a combinatorial interpretation for the case $\bfa=0$ is enough. In particular, if we denote by $A=\supp\bfa$, we have that
\begin{equation}
  \label{eq:TiOfLink}
  T^1_{\bfa - \bfb}(\D) = T^1_{-\bfb}(\lk_\D A).
\end{equation}
For the above reason we will use the following convention.
\begin{convention}
  Throughout this paper $\bfb$ will always denote a 0-1 vector, and we will use the same notation for its support. So, according to context, we may have 
  \[
    \bfb\in\set{0,1}^n \quad\text{or}\quad\bfb\subseteq [n].
    \]    
\end{convention}
\noindent To present the combinatorial characterization of $T^1_{-\bfb}(\D)$ from \cite{altmann2000stanleyreisner} we need the following.  
\begin{definition}
Let $\Delta$ be a simplicial complex on $[n]$ and $\bfb \subseteq [n]$. We define
\begin{align*}
    & \Ndel{\bfb}{\D} = \Set{F  \in \D \sodass F \cap \bfb = \emptyset,~~ F \cup \bfb \notin \D}  \text{ and }\\
    & \Ndelred{\bfb}{\D} = \Set{F  \in \Ndel{\bfb}{\D} \sodass \exists~ \bfb' \subsetneq \bfb \text{ with } F\cup \bfb' \notin \D}.
\end{align*}
\end{definition}
\begin{remark}
  \label{rem:Ndel}
  By the above  definition we have
  \[
    \Ndel{\bfb}{\D} =
    \begin{cases}
      \D\sdif \str_\D \bfb&\text{if~}\bfb\in\D,\\
      \D\sdif \bfb &\text{if~}\bfb\notin\D.      
    \end{cases}\\
  \]
\end{remark}
\noindent For every nonempty set $F\subset [n]$ one assigns the \defi{relatively open simplex}
 \[
   \textstyle
  \opensimplex{F} = \Set{\alpha:[n]\too[0,1]\sodass \sum_{i=1}^n\alpha(i)=1 \text{~and~} (\,\alpha(i)\neq 0 \iff i\in F\,)}.
\]
Each collection of subsets $\Gamma\subseteq 2^{[n]}$ determines thus a topological space in the following way.
\[
  \opensimplex{\Gamma}=
  \begin{cases}
\textstyle    \bigcup_{F\in\Gamma} \opensimplex{F}&\text{if~}\emptyset\notin \Gamma,\\[1ex]
\textstyle    \cone\left(\bigcup_{F\in\Gamma}\opensimplex{F}\right)&\text{if~}\emptyset\in \Gamma.\\    
  \end{cases}
\]
Many of our proofs rely  on the following theorem of Altmann and Christophersen.
\begin{theorem}[{\cite[Theorem 9]{altmann2000stanleyreisner}}]
  \label{thm:TiAsRelativeCohomology}
  Let \D\ be a simplicial complex on $[n]$ and $\bfb\in\Set{0,1}^n$, which we will identify  with its support. If $\card{\bfb}>1$, then $T^1_{-\bfb}(\D)$ is given by 
  \[
    T^1_{-\bfb}(\D) \simeq H^{0}(\opensimplex{\Ndel{\bfb}{\D}},\opensimplex{\Ndelred{\bfb}{\D}},\K)\qquad\text{for }.
  \]
  If $\# \bfb=1$, then the above formula holds if we use the reduced relative cohomology instead\,\footnote{~It is easy to see, that if $\#\bfb=1$, then $\Ndelred{\bfb}{\D}=\emptyset$. So one actually takes the relative cohomology of $\Ndel{\bfb}{\D}$, avoiding thus  reduced relative cohomology.}.   
\end{theorem}

\section{Upper bounds on the dimension of  $T^1$ for simplicial complexes}
\label{sec:T1inGeneral}
In this section we study the first cotangent module for abstract simplicial complexes. The main results are \Cref{prop:BoundOfDegreewiseDimensionOfFirstCCMOne}, which gives an upper  bound for the dimension of $T^1_{\bfa-\bfb}(\D)$ when $\bfb$ is supported on a face of $\D$, and \Cref{prop:dependentCotangentCohomology} which gives a full description of $T^1_{\bfa-\bfb}(\D)$ when $\bfb$ is supported on a nonface of \D.\\[1ex]
\indent We start by studying $\Ndel{\bfb}{\D}$ and $\Ndelred{\bfb}{\D}$. To this aim, let $A\in \Delta$ and $\bfb\in[n]$, with $\bfb\neq \emptyset$. Define the (unoriented) graph $\incgraph{A,\bfb}{\Delta}$ on the vertex set $\Ndel{\bfb}{\lk_\D A}$, where
\[
\Set{F_1, F_2}\text{~is an edge~}\iff F_1\subsetneq F_2\text{~~or~~}F_2\subsetneq F_1.
\]
Let $\incgraphred{A,\bfb}{\Delta}$ be the subgraph of $\incgraph{A,\bfb}{\Delta}$ that consists of all connected components that contain no vertex from  $\Ndelred{\bfb}{\link[\Delta]{A}}$. By Theorem~\ref{thm:TiAsRelativeCohomology} the graded components of $T^1$ are isomorphic to some relative cohomology module $H^0$. So  $T^1_{\bfa-\bfb}$ counts the number of connected components which do not intersect the given subspace. Thus, denoting $\supp \bfa  = A$,  we get from \cite[Theorem 9]{altmann2000stanleyreisner} that
\begin{equation}
  \label{eq:dimT1FromGraph}
  \dim_\K T^1_{\bfa-\bfb}(\D) =
  \begin{cases}
    \text{the number of connected components of $\incgraphred{A,\bfb}{\D}$}&\text{if~}\card{\bfb}>1,\\
    \text{the number of connected components of $\incgraphred{A,\bfb}{\D}$} - 1 &\text{if~}\card{\bfb}=1.
  \end{cases}
\end{equation}

\begin{lemma}\label{lemma:MinimalElementsofNdelAndNdelred}
  Let $\Delta$ be a simplicial complex on $[n]$ and $\bfb\subseteq [n]$. The inclusion-minimal elements of $\Ndel{\bfb}{\D}$ and $\Ndelred{\bfb}{\D}$ satisfy:
\begin{enumerate}[label=(\roman*),font=\upshape]
    \item\label{item:lme1}$\min_{\subseteq} \Ndel{\bfb}{\D} \subseteq \Set{C \sdif \bfb \sodass C \in \matroid{C}{\Delta},~ C \cap \bfb \neq \emptyset}$.
    \item\label{item:lme2} $\min_{\subseteq} \Ndelred{\bfb}{\D} \subseteq \Set{C \sdif \bfb \sodass C \in \matroid{C}{\Delta},~ C \cap \bfb \neq \emptyset,~ \bfb \nsubseteq C}$.    
\end{enumerate}
\end{lemma}

\begin{proof}
\noindent \ref{item:lme1} Let $X \in \Ndel{\bfb}{\D}$ be minimal under inclusion. The set $X \cup \bfb \notin \D$ contains some minimal nonface $C \in \matroid{C}{\Delta}$.
In particular $C\cap \bfb\neq\emptyset$, since $X \in \D$. 
Thus $C \sdif \bfb \subseteq X$ is a face with
\[
  (C\sdif \bfb)\cap \bfb = \emptyset \quad\text{and}\quad (C\sdif \bfb) \cup \bfb \notin \D.
\]
So by definition $C\sdif \bfb\in\Ndel{\bfb}{\D}$, and by the minimality of $X$, we must have $C\sdif \bfb=X$.\\[1ex]
\ref{item:lme2} If $X \in \Ndelred{\bfb}{\D}$ is minimal under inclusion, $X \cup \bfb' \notin \D$ for some $\bfb' \subsetneq \bfb$. Then there exists $C\in\matroid{C}{\D}$ with $C\subseteq X\cup \bfb'$. Because $X\cap \bfb=\emptyset$ and $\bfb'\subsetneq \bfb$, we cannot have $\bfb\subseteq C\subseteq X\cup \bfb'$. The rest follows by a similar argument as above.
\end{proof}

\begin{proposition}\label{prop:NdelrelEmptysetEquivalence}
Let $\Delta$ be a  simplicial complex on $[n]$ and $\bfb \subseteq [n]$. The following two conditions are equivalent:
\begin{enumerate}[label=(\roman*),font=\upshape, itemsep=5pt]
    \item\label{item:iProp32} $\Ndelred{\bfb}{\D} = \emptyset$.
    \item\label{item:iiProp32} For every $ C \in \matroid{C}{\Delta}$ the set $\bfb$ is either contained in or disjoint to $C$.
\end{enumerate}
Furthermore, the two conditions above imply that 
\begin{enumerate}[label=(\roman*),font=\upshape, itemsep=5pt,start=3]
    \item\label{item:iiiProp32} $\min_{\subseteq} \Ndel{\bfb}{\D} = \Set{C \sdif \bfb \sodass C \in \matroid{C}{\Delta}, \bfb \subseteq C} $.
    \end{enumerate}    
\end{proposition}

\begin{proof}
\ref{item:iProp32} $\Rightarrow$ \ref{item:iiProp32}. Assume there exists a $C \in \matroid{C}{\Delta}$ with $\bfb \nsubseteq C$ and $C \cap \bfb \neq \emptyset$. 
Then $C \sdif \bfb$ is a face which is disjoint to $\bfb$. 
Thus $C \sdif \bfb$ is contained in $\Ndel{\bfb}{\D}$.
Since $\bfb \nsubseteq C$ there exists $v \in \bfb$ such that $(C \sdif \bfb) \cup (\bfb \sdif \Set{v})$ is a nonface. Therefore, $C \sdif \bfb \in \Ndelred{\bfb}{\D}$ contradicting $\Ndelred{\bfb}{\D} = \emptyset$.\\[1ex]
\ref{item:iiProp32} $\Rightarrow$ \ref{item:iProp32}. Assume $\Ndelred{\bfb}{\D}\neq\emptyset$. This implies that it contains a minimal element.  \Cref{lemma:MinimalElementsofNdelAndNdelred}\,\ref{item:lme2} implies thus that
\[
  \Set{C \sdif \bfb \sodass C \in \matroid{C}{\Delta},~ C \cap \bfb \neq \emptyset,~ \bfb \nsubseteq C}\neq\emptyset
\]
which contradicts \ref{item:iiProp32}.\\[1ex]
\ref{item:iiProp32} $\Rightarrow$ \ref{item:iiiProp32}.  By Lemma~\ref{lemma:MinimalElementsofNdelAndNdelred}\,\ref{item:lme1} every minimal Element in $\Ndel{\bfb}{\D}$ is of the form $C\sdif \bfb$ with $C\in\matroid{C}{\D}$ and $C\cap \bfb\neq \emptyset$. 
By \ref{item:iiProp32} it follows from $C \cap \bfb \neq \emptyset$ that $\bfb \subseteq C$, and we have the direct inclusion by the same lemma.
For the other inclusion let $C \in \matroid{C}{\Delta}$ be a minimal nonface with $\bfb \subseteq C$.
By definition, we have that $C \sdif \bfb \in \Ndel{\bfb}{\D}$. 
To see that $C \sdif \bfb$ is also minimal assume there exists $A \in \Ndel{\bfb}{\D}$ with $A \subsetneq C \sdif \bfb$. 
This implies that  $A \cup \bfb \notin \D$, with $A\cup \bfb\subsetneq C$, which contradicts that $C$ is a minimal nonface.
\end{proof}
\begin{remark}
  \label{rem:iii=/=>iInProp}
  The implication \ref{item:iiiProp32}$\Rightarrow$\ref{item:iProp32} in Proposition~\ref{prop:NdelrelEmptysetEquivalence} does not usually hold. For instance, if \D\ is the 1-skeleton of the tetrahedron on $\Set{1,2,3,4}$ and if $\bfb=12:=\Set{1,2}$ then
  \begin{eqnarray*}
    \Ndel{\bfb}{\D}           & = & \Set{3,4,34},                           \\
    \Ndelred{\bfb}{\D}        & = & \Set{34}~~\neq~~\emptyset,~~\text{but~} \\
\min_\subseteq\Ndel{\bfb}{\D}  & = & \Set{3,4}~~=~~\Set{C\sdif \bfb\sodass C\in\matroid{C}{\D}, \bfb\subseteq C}. 
  \end{eqnarray*}    
\end{remark}

\begin{definition}
  Let $\Delta$ be a  simplicial complex on $[n]$ and $\bfb \subseteq [n]$. We define $\matroidrel{C}{\Delta}{\bfb}$ as the set of minimal nonfaces containing $\bfb$:
  \[\matroidrel{C}{\Delta}{\bfb}  \definedas \Set{C \in \matroid{C}{\Delta} \sodass \bfb \subseteq C}.\]
\end{definition}

\begin{lemma}\label{lemma:minimalelementsofmatroidalNdel}
  Let $\Delta$ be a simplicial complex and  $\bfb \subseteq [n]$ be an arbitrary set  such that for every minimal nonface $C\in \matroid{C}{\D}$ we have either $\bfb\cap C=\emptyset$  or $\bfb\subseteq C$.
  The following map  is a bijection:
  \[
    \varphi: \matroidrel{C}{\Delta}{\bfb} \longrightarrow \min_{\subseteq} \Ndel{\bfb}{\D},\quad \ C \mapsto C \sdif \bfb.
  \] 
We can extend $\varphi$ to a surjection from minimal nonfaces containing $\bfb$ to graph-components:
  \[
    \phi: \matroidrel{C}{\Delta}{\bfb} \longrightarrow
    \pathcomponents{\incgraphred{\emptyset,\bfb}{\Delta}},\quad \ C \mapsto [C \sdif \bfb].
  \] 
\end{lemma}

\begin{proof}
  It follows from Proposition~\ref{prop:NdelrelEmptysetEquivalence} that the codomain of $\varphi$ coincides with the set
  \[
  \min_{\subseteq} \Ndel{\bfb}{\D}=  \Set{C \sdif \bfb \sodass C \in \matroidrel{C}{\Delta}{\bfb},~~ \bfb\subseteq C}.
  \]  
  Thus the first map  is bijective since $\bfb$ is fully contained in every $C \in \matroidrel{C}{\Delta}{\bfb}$.  
Therefore, we can extend $\varphi$ by sending elements of $\min_{\subseteq} \Ndel{\bfb}{\D}$ to their graph-component in $\incgraph{\emptyset,\bfb}{\Delta}$. 
From \Cref{prop:NdelrelEmptysetEquivalence} it also follows that $\Ndelred{\bfb}{\D}$ is the empty set. 
Thus $\incgraph{\emptyset,\bfb}{\Delta} = \incgraphred{\emptyset,\bfb}{\Delta}$ and we can choose $\pathcomponents{\incgraphred{\emptyset,\bfb}{\Delta}}$ as the codomain of the extension. The map $\phi$ is surjective since any component contains at least one element in $\min_{\subseteq} \Ndel{\bfb}{\D}$.
\end{proof}

\begin{proposition}\label{prop:BoundOfDegreewiseDimensionOfFirstCCMOne}
For every simplicial complex \D\ and every $\bfb \in \D$ we have
\begin{align*}
  \dim  \firstcotangentcohomology{-\bfb}{\Delta} \leqslant \min
  \Set{%
  \card{\big(\,\matroid{C}{\lk_\D \bfb} \cap (\D\sdif \bfb)\,\big)},~~
  \card{\big(\,\matroid{B}{\D\sdif \bfb} \sdif \lk_\D \bfb\,\big)}}.
\end{align*}
If $\card{\bfb} = 1$, then 1 may be subtracted from the right-hand side when the latter is positive. 
\end{proposition}

\begin{proof}
  Every connected component of $\incgraph{\emptyset,\bfb}{\Delta}$ has at least one vertex that is inclusion-minimal in $\Ndel{\bfb}{\D}$ and one vertex that is inclusion-maximal in $\Ndel{\bfb}{\D}$.  From (\ref{eq:dimT1FromGraph}) we get thus
  \[
    \dim  \firstcotangentcohomology{-\bfb}{\Delta} \leqslant \min \Set{\text{\# of minima in $\incgraph{\emptyset,\bfb}{\Delta}$\,,~~\# of  maxima  in $\incgraph{\emptyset,\bfb}{\Delta}$}}.
  \]
  Recall that by \Cref{rem:Ndel}, as $\bfb\in \D$, we have 
  \(
    \Ndel{\bfb}{\D} =  \D\sdif\str_\D \bfb.
  \)
  This means that 
  \[
F\in\Ndel{\bfb}{\D} \iff 
F\in \D\sdif \bfb~~\text{and}~~F\notin \lk_\D \bfb.
\]
If $X$ is minimal in $\Ndel{\bfb}{\D}$, then
for every subset $X'\subsetneq X$ we have $X'\notin\Ndel{\bfb}{\D}$. As $X'\in \D\sdif \bfb$ still holds, we get $X'\in\lk_\D \bfb$. This means that
  \[
    X\text{~is minimal in~}\Ndel{\bfb}{\D} \iff X\in \matroid{C}{\lk_\D \bfb} \cap (\D\sdif \bfb).
  \]
  If $Y$ is  maximal in $\Ndel{\bfb}{\D}$, then every $Y'$ with $Y\subsetneq Y'$ we have $Y'\notin\Ndel{\bfb}{\D}$. As $Y'$ still fulfills $Y'\notin\lk_\D \bfb$, we get $Y'\notin \D\sdif \bfb$. Thus $Y$ is a facet of $\D\sdif \bfb$.
  This means that
  \[
    Y\text{~is maximal in~}\Ndel{\bfb}{\D}
    \iff
    Y\in\matroid{B}{\D\sdif \bfb}\sdif\lk_\D \bfb.
  \]
  \vspace{-1em}
\end{proof}
Using the definitions we can restate this bound as follows.
\begin{corollary}\label{prop:BoundOfDegreewiseDimensionOfFirstCCM}
For every simplicial complex $\Delta$ and every $\bfb \in \D$ we have
\begin{align*}
  \dim  \firstcotangentcohomology{-\bfb}{\Delta}
  \leqslant
  \min \Set{\card{\big(\,\matroid{C}{\lk_\D \bfb} \sdif \matroid{C}{\Delta \sdif \bfb}\,\big)},~~ \card{\big(\,\matroid{B}{\Delta\sdif \bfb} \sdif \matroid{B}{\lk_\D \bfb}\,\big)} }\!.
\end{align*}
\end{corollary}
\noindent We denote by $\partial F= \Set{A\subsetneq F}$ the boundary of $2^F$ as an abstract simplex.
\begin{proposition}\label{prop:dependentCotangentCohomology}
Let $\Delta$ be a simplicial complex on $[n]$ and $\bfb\in 2^{[n]}\sdif \D$ be a nonface. Then
\begin{equation*}
    \dim \firstcotangentcohomology{-\bfb}{\Delta} = 
    \begin{cases}
        1 &\quad 
        \text{if } \Delta \cong (\Delta \sdif \bfb) * \partial \bfb \text{ and } \card{\bfb} > 1,\\
        0  &
        \quad\text{if otherwise. }
    \end{cases}
\end{equation*}
\end{proposition}

\begin{proof}
 \Cref{rem:Ndel} yields that $\Ndel{\bfb}{\D} =\D\sdif \bfb $ and therefore $\incgraph{\emptyset,\bfb}{\Delta}$ is connected. So if $\card{\bfb} = 1$, then $\dim \firstcotangentcohomology{-\bfb}{\Delta} = 0$.
If $\card{\bfb} > 1$, then $\dim \firstcotangentcohomology{-\bfb}{\Delta} = 1$ if and only if $\Ndelred{B}{\D}  = \emptyset$. Otherwise, we have $\dim \firstcotangentcohomology{-\bfb}{\Delta} = 0$. 
By \Cref{prop:NdelrelEmptysetEquivalence} this only happens if $\bfb$ is either fully contained in or disjoint to any minimal nonface of $\Delta$.
Since $\bfb$ is a nonface this means that $\bfb$ itself is a minimal nonface disjoint to any other minimal nonface.
Thus
\[
  \matroid{C}{\Delta} ~=~\matroid{C}{\Delta \sdif \bfb} \cup \Set{\bfb} ~=~ \matroid{C}{\Delta \sdif \bfb} \cup \matroid{C}{\partial \bfb} ~=~ \matroid{C}{(\Delta \sdif \bfb) * \partial \bfb}.
\]
\end{proof}
\section{Computing $T^1$ for matroids}
\label{sec:t1forMatroids}
It turns out that for a matroidal Stanley-Reisner ring $\K[\M ]$ one can compute the dimensions of the graded pieces of $T^1$ by looking at circuits. In a certain sense, simplicial complexes which are not matroids have ``too many'' circuits.
This makes the sets $\Ndelred{\bfb}{\D}$  more complicated than in the nice case of matroids.
These two observations lead to the main result of this section: \Cref{thm:FirstCotangentCohomologyMatroid}.
In Corollary~\ref{FirstCotangentFormulaDeterminesMatroid} we show it is enough to look at $T^1$ in degrees $-e_i$ to determine whether a simplicial complex is a matroid.
We start with two results which are specific to matroids.

\begin{lemma}\label{lem:MaximalElementsOfMatroidalNdelAndNdelred}
Let $\M $ be a matroid, $\bfb \subseteq [n]$ an arbitrary set, and $B_1,B_2 \in \matroid{B}{\M  \sdif \bfb}$ bases of the deletion of $\bfb$. For any $\bfb' \subseteq \bfb$ we get the equivalence
\[
    B_1 \cup \bfb' \in \M  \iff B_2 \cup \bfb' \in \M .
\]
\end{lemma}

\begin{proof}
  We can assume without loss of generality that $\bfb' = \bfb$ and prove the statement by induction over the cardinality $\card{\bfb}$. Considering the symmetry of the conclusion, we only show the implication 
\begin{equation}\label{eq:IndependentSetUnionImplication}
B_1 \cup \bfb \in \M  \implies B_2 \cup \bfb \in \M .
\end{equation}
If $\bfb = \emptyset$, then we obtain an implication between two tautologies. Assume now that $\card{\bfb}>0$ and that (\ref{eq:IndependentSetUnionImplication})  holds for any subset of $[n]$ of cardinality $(\card{\bfb})-1$.
If $B_1 \cup \bfb \in \M $, then, since $\card{B_2}=\card{B_1}<\card{(B_1\cup\bfb)}$,  we can use the independent set exchange axiom to find a $v \in B_1 \cup \bfb$, with $v\notin B_2$, such that $B_2 \cup \Set{v} \in \M $.
If $v\notin \bfb$, then $B_2 \cup \Set{v}$ would be an independent set of $\M \sdif \bfb$ that properly contains a basis. So both $B_1 \cup \Set{v}$ and $B_2 \cup \Set{v}$ with $v\in\bfb$ are  bases of $\M \sdif (\bfb\sdif \Set{v})$. Since $\card{(\bfb\sdif\set{v})}=(\card{\bfb})-1$, the induction hypothesis implies (\ref{eq:IndependentSetUnionImplication}).
\end{proof}

\begin{corollary}\label{coro:NdelredOfMatroids}
Let $\M $ be a matroid and $\bfb \subseteq [n]$. 
\begin{enumerate}[label=(\roman*),font=\upshape]
\item\label{item:lem42.1} If $\Ndel{\bfb}{\M } \neq \emptyset$, then $\matroid{B}{\M  \sdif \bfb} \subseteq \Ndel{\bfb}{\M }$.
\item\label{item:lem42.2} If $\Ndelred{\bfb}{\M } \neq \emptyset$, then $ \matroid{B}{\M  \sdif \bfb} \subseteq \Ndelred{\bfb}{\M }$.
\item\label{item:lem42.3} If $\Ndelred{\bfb}{\M } \neq \emptyset$, then the inclusion-maximal elements in $\Ndel{\bfb}{\M }$ and $\Ndelred{\bfb}{\M }$ are the same.
\end{enumerate}
\end{corollary}

\begin{proof}
\ref{item:lem42.1} Let $F\in\Ndel{\bfb}{\M }$. This means $F \in \M \sdif \bfb$ and $F \cup \bfb \notin \M $.  Thus, all independent sets $X$ that contain  $F$ have the property that $X \cup \bfb \notin \M $. In particular, all bases of $\M \sdif\bfb$ that contain $F$ are in $\Ndel{\bfb}{\M }$.
By
\Cref{lem:MaximalElementsOfMatroidalNdelAndNdelred} all bases of $\matroid{B}{\M \sdif \bfb}$ must also be in $\Ndel{\bfb}{\M }$.\\[1ex]
\ref{item:lem42.2} If $F\in\Ndelred{\bfb}{\M }$, then there exists $\bfb'\subsetneq\bfb$ such that $F\cup \bfb'\notin \M $. We then conclude by the same argument as in the previous point.  \\[1ex]
\ref{item:lem42.3} We have $\Ndelred{\bfb}{\M }\subseteq\Ndel{\bfb}{\M }\subseteq \M \sdif \bfb$, with the first two  closed under taking supersets within the third. So the set of  maximal elements in each is exactly the intersection with $\matroid{B}{\M \sdif \bfb}$ and we conclude by the previous two points.
\end{proof}

\begin{lemma}\label{lemma:forgottenaxiom}
If $\Delta$ is a nonempty  simplicial complex, then the following are equivalent.
\begin{enumerate}[label=(\roman*),font=\upshape,
leftmargin = 2.5em, labelwidth=2em, itemindent = 0em, labelindent = 0em, listparindent = 3em, itemsep =1ex  ]
    \item\label{item:fax1} $\D$ is a matroid.
    \item\label{item:fax2} For all $\bfb\subseteq[n]$, if $\bfb\cap C \in\Set{\emptyset, \bfb}$ for all $C\in\matroid{C}{\D}$, then every element of $\Ndel{\bfb}{\D}$ contains a unique inclusion-minimal  element of $\Ndel{\bfb}{\D}$.        
    \item\label{item:fax3} For all $v \in [n]$, any element of $\Ndel{v}{\D}$ contains a unique inclusion-minimal  element of $\Ndel{v}{\D}$.   
\end{enumerate}
\end{lemma}

\begin{proof}
\ref{item:fax1} $\Rightarrow$ \ref{item:fax2} The statement is trivial for $\bfb=\emptyset$.
If $\bfb \notin \Delta$,
then by \Cref{rem:Ndel} we get $\Ndel{\bfb}{\D} = \D \sdif \bfb$ which contains a unique inclusion-minimal set: $\emptyset$. 
Thus, it suffices to consider the case  $\bfb\in\D$.
 So for all circuits $C \in \matroid{C}{\Delta}$ either $\bfb \subsetneq C$ or $\bfb \subseteq [n] \sdif C$.
By \Cref{prop:NdelrelEmptysetEquivalence} we have 
\[
  \min_{\subseteq} \Ndel{\bfb}{\D} = \Set{C \sdif \bfb \sodass C \in \matroid{C}{\Delta},~~ \bfb \subseteq C}.
\]
Assume that $F \in \Ndel{\bfb}{\D}$ contains two distinct minimal elements of $\Ndel{\bfb}{\D}$: $C_1\sdif \bfb$ and $C_2\sdif \bfb$.  Because $(C_1\sdif\bfb)\cup (C_2\sdif \bfb) \subseteq F\in\D$ and  the union with $\bfb$ contains circuits, we have that
\[
  (C_1\sdif\bfb)\cup (C_2\sdif \bfb) \in \Ndel{\bfb}{\D}.
\]
Let $v\in\bfb\subseteq C_1\cap C_2$.  
By the circuit exchange axiom~\ref{item:circEl1} we find a circuit $C_3 \subseteq (C_1 \cup C_2 )\sdif \set{v}$. As $v\notin C_3$, we have $\bfb\nsubseteq C_3$. We also have  $C_3 \nsubseteq (C_1\sdif\bfb) \cup (C_2\sdif\bfb)$ since the latter is a face of \D. So we must also have $C_3\cap\bfb\neq\emptyset$.
This contradicts that all circuits of $\Delta$ are either disjoint to or contain $\bfb$.\\[1ex]
\ref{item:fax2} $\Rightarrow$ \ref{item:fax3} All singletons are either fully contained in or disjoint to a set.\\[1ex]
\ref{item:fax3} $\Rightarrow$ \ref{item:fax1} Assume that $\Delta$ is not a matroid. Then there are at least two minimal nonfaces $C_1, C_2 \in \matroid{C}{\Delta}$ and an element $c \in C_1 \cap C_2$ such that there exists no minimal nonface contained in $(C_1 \cup C_2) \sdif \set{c}$.
In other words $(C_1 \cup C_2) \sdif \set{c}$ is a face. 
But then $(C_1 \cup C_2) \sdif \set{c} \in \Ndel{c}{\D}$ does contain two sets $C_1 \sdif\set{c}, C_2 \sdif\set{c} \in \Ndel{c}{\D}$. Both $C_1 \sdif\set{c}$ and $C_2 \sdif\set{c}$ are minimal under inclusion by \Cref{prop:NdelrelEmptysetEquivalence}.\ref{item:iiiProp32} and don't contain each other.
\end{proof}

\begin{remark}
The implication \ref{item:fax1} $\Rightarrow$ \ref{item:fax3} is given in \cite[Proposition 1.1.6]{OXLEY}. It turns out that this implication is an equivalence.
\end{remark}

\begin{proposition}\label{prop:MatroidalIncgraphBijection}
Let $\Delta$ be a simplicial complex on $[n]$. The following are equivalent.
\begin{enumerate}[label=(\roman*),font=\upshape]
    \item\label{item:pmi1} $\Delta$ is a matroid.
    \item\label{item:pmi2} For every $\bfb \subseteq [n]$ such that $\Ndelred{\bfb}{\D} = \emptyset$ the minimal elements of $\Ndel{\bfb}{\D}$ with respect to inclusion are unique in their connected component of $\incgraph{\emptyset,\bfb}{\Delta}$.
    \item\label{item:pmi3} For every $v \in [n]$ the minimal elements of $\Ndel{v}{\D}$ with respect to inclusion are unique in their connected component of $\incgraph{\emptyset,v}{\Delta}$.
    \item\label{item:pmi4} For every nonempty $\bfb \subseteq [n]$ such that $\Ndelred{\bfb}{\D} = \emptyset$ the following map to  is injective
      \[
        \phi: \matroidrel{C}{\Delta}{\bfb} \longrightarrow \pathcomponents{\incgraphred{\emptyset,\bfb}{\Delta}},\quad \ C \mapsto [C \sdif \bfb].
        \]
      \item\label{item:pmi5} For every nonempty $\bfb \subseteq [n]$ such that $\Ndelred{\bfb}{\D} = \emptyset$ the following map is bijective
        \[
          \phi: \matroidrel{C}{\Delta}{\bfb} \longrightarrow \pathcomponents{\incgraphred{\emptyset,\bfb}{\Delta}},\quad \ C \mapsto [C \sdif \bfb].
        \]          
\end{enumerate}
\end{proposition}

\begin{proof}
  \ref{item:pmi1} $\Rightarrow$ \ref{item:pmi2} Assume there exists a $\bfb\in [n]$ with $\Ndelred{\bfb}{\D}\neq \emptyset$ such that there are different minimal elements  in some component of $\incgraph{\emptyset,\bfb}{\Delta}$. By \Cref{lemma:MinimalElementsofNdelAndNdelred} these have the form $C_1\sdif \bfb$ and $C_2\sdif\bfb$ with $C_1,C_2\in\matroid{C}{\D}$ and $C_1\cap \bfb\neq \emptyset\neq C_2\cap \bfb$.
Every path in $\incgraph{\emptyset, \bfb}{\Delta}$ connecting $C_1\sdif \bfb$ and $C_2\sdif\bfb$ can be modified such that each inclusion-ascending chain ends at a maximal element, and each inclusion-descending chain ends at a minimal element. Thus, in any connected component with distinct minimal elements, there is at least one maximal set containing two minimal sets.
Therefore, without loss of generality, we can assume $C_1\sdif\bfb$ and $C_2\sdif\bfb$ to be two minimal sets  contained in a common set $F\in\Ndel{\bfb}{\D}$.
As \D\ is a matroid and $\Ndelred{\bfb}{\D} = \emptyset$, all sets in $\Ndel{\bfb}{\D}$ contain a unique minimal set by \Cref{lemma:forgottenaxiom}~\ref{item:fax2} - a contradiction. \\[1ex]
\ref{item:pmi2} $\Rightarrow$ \ref{item:pmi3} Singletons  are all either contained in or disjoint to sets.\\[1ex]
\ref{item:pmi3} $\Rightarrow$ \ref{item:pmi1} Note that the inclusion-minimal elements of the graph $\incgraph{\emptyset,v}{\Delta}$ are exactly the minimal elements of  $\Ndel{v}{\D}$. Thus, the implication follows from  \Cref{lemma:forgottenaxiom}.\\[1ex]
\ref{item:pmi2} $\Leftrightarrow$ \ref{item:pmi4} The map $\phi$ sends elements of $\matroidrel{C}{\Delta}{\bfb}$ to the connected component of minimal elements of $\Ndel{\bfb}{\D}$. Therefore, $\phi$ being injective means exactly that those elements are unique.\\[1ex]
\ref{item:pmi4} $\Leftrightarrow$ \ref{item:pmi5} The map $\phi$ is always surjective by \Cref{lemma:minimalelementsofmatroidalNdel}.
\end{proof}

\begin{remark}
Without much effort one could also prove that the conditions \ref{item:pmi4} and \ref{item:pmi5} of \Cref{prop:MatroidalIncgraphBijection} have equivalent variants with singletons instead of sets $\bfb \subseteq [n]$ such that $\Ndelred{\bfb}{\D} = \emptyset$.
\end{remark}
 
Having considered for any matroid $\M$  both  the case $\Ndelred{\bfb}{\M } = \emptyset$ and the case $\Ndelred{\bfb}{{\M }} \neq \emptyset$, we can prove the first theorem of this paper. 

\begin{theorem}\label{thm:FirstCotangentCohomologyMatroid}
Let $\D$ be a simplicial complex on $[n]$. Then \D\ is a matroid if and only if the following  formula holds for all $\bfa-\bfb \in \Z^n$ with $\bfa \in \mathbb{\Z}_{\geqslant 0}^{n}$ and $\bfb \in\set{0,1}^{n}$ having disjoint support:
\begin{equation}\label{eq:firstcotangentcohomologyofmatroids}
\dim \firstcotangentcohomology{\bfa-\bfb}{\D} =
\begin{cases}
     0
     & \text{if~} A \notin \D \text{~or~} \bfb = \emptyset \text{~or~}  \Ndelred{\bfb}{{\D}} \neq \emptyset, \\[1ex]
     \max\Set{\card{\matroidrel{C}{\lk_\D\!A}{\bfb}} -1,0}
     & \text{if $\card{\bfb} = 1$},     \\[1ex]
     \card{\matroidrel{C}{\lk_\D\!A}{\bfb }}
     & \text{if otherwise,}
\end{cases}
\end{equation}
where $A = \supp(\bfa)$,  $\bfb$ is identified with  $\supp(\bfb)$ and $\matroid{C}{\lk_\D\!A}{(\bfb)}=\Set{C\in\matroid{C}{\lk_\D\!A}\sodass \bfb \subseteq C}.$
\end{theorem}

\begin{proof}
  We prove first, that if $\D$ is a matroid, then the formula holds.
  By \Cref{lem:L2AC} we only have to consider the case $A\in \D$, in which case we have by (\ref{eq:TiOfLink}) that
  \[
    \firstcotangentcohomology{\bfa-\bfb}{\D} \cong \firstcotangentcohomology{-\bfb}{\lk_\D\!A}.
  \]  
As the link of an independent set in a matroid is still a matroid,  we can restrict our proof to the case where $A = \emptyset$.
Also, if $\bfb=\emptyset$, then  $\firstcotangentcohomology{\bfa-\bfb}{\D} = 0$.
We already dealt with the setting where $\bfb \notin \D$ in \Cref{prop:dependentCotangentCohomology}, and a direct check gives the same formula. 
Thus, it remains to prove the theorem for  $A = \emptyset$ and $\emptyset \neq \bfb \in {\D}$. 
By (\ref{eq:dimT1FromGraph}) it is enough to prove that
\[
   \text{the number of connected components of $\incgraphred{\emptyset, \bfb}{\D}$} = 
    \begin{cases}
      \card{\matroid{C}{\D}{(\bfb)}}      &  \text{if~} \Ndelred{\bfb}{{\D}} = \emptyset,\\[1ex]
        0&\text{if~ otherwise.}
    \end{cases}
  \]
If $\Ndelred{\bfb}{{\D}} \neq \emptyset$, then by \Cref{coro:NdelredOfMatroids}~\ref{item:lem42.3} the inclusion-maximal elements in $\incgraph{\emptyset,\bfb}{\D}$ are all contained in $\Ndelred{\bfb}{\D}$. This means that every connected component contains something from $\Ndelred{\bfb}{\D}$ and
thus $\incgraphred{\emptyset, \bfb}{\D}$ is the empty graph.
If $\Ndelred{\bfb}{{\D}} = \emptyset$, by \Cref{prop:MatroidalIncgraphBijection}~\ref{item:pmi5} we have a bijection between $\matroidrel{C}{\D}{\bfb}$ and the connected components of $\incgraphred{\emptyset,\bfb}{\D}$. \\[1ex]
To see that the formula fails otherwise, assume that $\Delta$ is a nonempty simplicial complex that is not a matroid. Note that, $\Ndel{v}{\D}=\emptyset$ if and only if $v$ is a coloop, and that  $\Ndel{v}{\D}=\D\sdif v$ if and only if $v$ is a loop. 
Therefore, by \Cref{prop:MatroidalIncgraphBijection}~\ref{item:pmi3}, there exists at least one $v \in [n]$ such that $\Ndel{v}{\D}\neq \emptyset$ and there are at least two distinct minimal elements of $\Ndel{v}{\D}$ in the same connected component of $\incgraph{\emptyset, v}{\D}$.  
Thus,  $\matroidrel{C}{\D}{v}\geqslant 2$ and 
\[
  \dim \firstcotangentcohomology{-e_v}{\Delta} \leqslant \card{\pathcomponents{\incgraph{\emptyset, v}{\Delta}}} - 1 < \card{\matroidrel{C}{\Delta}{v}} -1 =\max\Set{\card{\matroidrel{C}{\D}{v}} -1,0}.
\]  
\end{proof}

From the proof of Theorem~\ref{thm:FirstCotangentCohomologyMatroid} we obtain the following.
\begin{corollary}
  \label{FirstCotangentFormulaDeterminesMatroid}
  A simplicial complex \D\ on $[n]$ is a matroid if and only if 
  \[    
    \dim\firstcotangentcohomology{-e_i}{\Delta} = \max \{ \card{\matroid{C}{\D}{(i)}}-1,0\}\quad\forall~ i=1,\dots,n.
    \]
  \end{corollary}

 In the final part of this section we will show that for matroids the dimension of each nonvanishing graded component of $T^1$ reaches the upper bound given in~\Cref{prop:BoundOfDegreewiseDimensionOfFirstCCM}.

\begin{proposition}\label{prop:firstcotangentcohomologyothergenerators}
Let $\M $ be a matroid, $A \in \M $, and denote by $\Lambda=\lk_\M  A$, the link of $A$ in $\M $. Let  $\bfb\in\Lambda$ which is either contained in or disjoint to any circuit of $\Lambda$. 
The following map is a bijection
\[
  \begin{tikzcd}[column sep  =2em,%
      row sep =0pt,%
      /tikz/column 1/.append style={anchor=base east},%
      /tikz/column 2/.append style={anchor=base west}
      ]  
      \matroidrel{C}{\Lambda}{\bfb}\ar[r]&
      \matroid{C}{\lk_{\Lambda}\bfb} \sdif
      \matroid{C}{\Lambda\sdif \bfb}\\
    C\ar[mapsto, r]& C\sdif \bfb.
  \end{tikzcd}
\]

\end{proposition}
    
\begin{proof}
  We can restrict to the case where $A = \emptyset$, because a matroid remains a matroid after contraction. So we may assume $\Lambda = \M $.
  Due to the assumption that $\bfb\cap C\in\Set{\emptyset, \bfb}$ for all $C\in\matroid{C}{\M }$, we may use \Cref{lemma:minimalelementsofmatroidalNdel}  
  we have that $\matroidrel{C}{\Lambda}{\bfb}$ is in bijection with the minimal elements of $\Ndel{\bfb}{\M }$.
  By \Cref{prop:BoundOfDegreewiseDimensionOfFirstCCM} these are exactly
  $ \matroid{C}{\lk_{\Lambda}\bfb} \sdif
  \matroid{C}{\Lambda\sdif \bfb}$.
\end{proof}

\begin{remark}
  Using the notation $\Lambda=\lk_\M  A$  of \Cref{prop:firstcotangentcohomologyothergenerators}, we can interchange $\matroidrel{C}{\Lambda}{\bfb}$ with $\matroid{C}{\lk_\Lambda \bfb} \sdif \matroid{C}{\Lambda\sdif \bfb}$ in \cref{eq:firstcotangentcohomologyofmatroids}. 
  This  means that, when a graded component of the first cotangent cohomology is nonzero, then its dimension is maximal in the sense of \Cref{prop:BoundOfDegreewiseDimensionOfFirstCCM}. However, while for matroids the existence of a circuit $C$ with $\bfb \cap C \notin \{\emptyset,\bfb\}$ implies $ \firstcotangentcohomology{-\bfb}\M=0$, for nonmatroids this is not the case. Take for example the simplicial complex $\D$ on $\{1,\dots,5\}$ with minimal nonfaces $\{12,13,234,235,145\}$\footnote{ Where $12$ denotes $\{1,2\}$ and so on.} and $\bfb=45$.  Then we have $\incgraph{\emptyset,\bfb}{\Delta} = \{[1],[2, 3, 23]\}$, where the square brackets indicate the connected component,  and $\Ndelred{\bfb}{\Delta} = \{2,3\}$. Thus, $\dim \firstcotangentcohomology{-\bfb}{\Delta} = 1$.
\end{remark}

\begin{example}
  Let $\uniform{n}{k}$ be an uniform matroid on the  ground set $[n]$, and let $\bfa-\bfb \in \Z^n$ with $\bfa \in \mathbb{N}^n$ and $\bfb \in \Set{0,1}^n$ with disjoint supports. Let $A =\supp \bfa$ and identify $\bfb$ and $\supp \bfb$. Using  $\matroid{C}{\uniform{n}{k}} = \Set{ C \subseteq [n] \sodass \card{C} = k+1}$ in \Cref{thm:FirstCotangentCohomologyMatroid} one obtains
\begin{equation}
\dim \firstcotangentcohomology{\bfa-\bfb}{{\uniform{n}{k}}} =
\begin{cases}
     1 
     & \text{if } \card{\bfb} > 1 \text{, } k = n-1 \text{ and } \card{A} \leqslant k,
     \\[1ex]
\displaystyle
     \binom{n - \card{A} - 1}{n-k-1}-1 
     & \text{if } \card{\bfb} = 1 \text{, } k < n-1  \text { and } \card{A} + \card{\bfb} \leqslant k ,
     \\[1ex]
     0  
     & \text{if otherwise}.
\end{cases}
\end{equation}
\end{example}

\section{Characterization of matroids by their first cotangent cohomology}\label{First cotangent cohomology as a matroid invariant}
So far we have seen that the dimensions of the graded components of  $T^1$ for matroids can be computed by counting circuits.
This count can be done on the nonfaces of any simplicial complex, but it returns $T^1$ if and only if the complex is a matroid. We ask next  how much information the first cotangent cohomology module of a matroid retains. In this section  we show that,  in contrast to the general case, matroids can be fully recovered from $T^1$ unless this module is trivial (\Cref{prop:reconstructrankonematroids}) and that triviality happens only in the extremal case of discrete matroids (\Cref{cor:trivialcotangentcohom.}).

\begin{lemma}\label{lemma:cotangentloopsandcoloops}
Let $\M $ be a matroid on $[n]$ and $v \in [n]$. The following are equivalent.
\begin{enumerate}[label=(\roman*),font=\upshape]
    \item\label{item:loops1} $v$ is a loop or a coloop.
    \item\label{item:loops2} $\firstcotangentcohomology{\bfc}{\M } = 0$ for any $\bfc\in \Z^{n}$ with $\bfc_v=-1$.
    \item \label{item:loops3} $\firstcotangentcohomology{-e_v - e}{\M } = 0$ for any $e \in \Set{0} \cup \Set{e_i \sodass i \in [n]~~\and~~ i\neq v}$.
\end{enumerate}
\end{lemma}

\begin{proof}
\ref{item:loops1} $\Rightarrow$ \ref{item:loops2}
Let $\bfb$ and $A$ be the support of the negative and positive parts of $\bfc$, respectively.
(Co)loops under restriction and deletion remain (co)loops. So if $v$ is a (co)loop in $\M $ and $\bfc_v=-1$, then $v$ is a (co)loop in $\lk_\M  A$. Therefore, we may use  $T^1_{\bfc}(\M )=T^1_{-\bfb}(\lk_\M  A)$. If $v\in\bfb$  is a (co)loop, then the only circuit that may contain $\bfb$ is $\Set{v}$.
Thus, the statement follows from \Cref{thm:FirstCotangentCohomologyMatroid}.\\[1ex]
\ref{item:loops2} $\Rightarrow$ \ref{item:loops3}
The third point is a restriction of the second  to a subset of degrees.\\[1ex]
\ref{item:loops3} $\Rightarrow$ \ref{item:loops1}
From $\firstcotangentcohomology{-e_v}{\M } = 0$  we obtain by \Cref{thm:FirstCotangentCohomologyMatroid} that $v$ is contained in at most one circuit.
If  $v$ is neither a loop nor a coloop, then $v$ is \emph{properly} contained in precisely one circuit $C$. So there exists $i\in C$ with $i\neq v$.
To obtain the contradiction $\dim T^{1}_{-e_v-e_i}\neq 0$ from \Cref{thm:FirstCotangentCohomologyMatroid}, we must make sure that $\Ndelred{\Set{v,i}}{\M }=\emptyset$. Let us assume $\Ndelred{\Set{v,i}}{\M }\neq\emptyset$. By~\Cref{prop:NdelrelEmptysetEquivalence} this means that there is a circuit $C'\neq C$ which intersects $\Set{v,i}$ in a proper subset. As $v$ is contained in only one circuit, we have $i\in C'$ and $v\notin C'$.
Using the strong circuit elimination axiom \cite[Proposition~1.4.12]{OXLEY} we find a circuit $C''$ which contains $v$ and which is contained in $(C \cup C') \sdif \Set{i}$.
This contradicts that only one circuit contains $v$. Thus $v$ is either a loop or a coloop.
 \end{proof}

\begin{corollary}\label{cor:trivialcotangentcohom.}
Let $\M $ be a matroid. The following are equivalent:
\begin{enumerate}[label=(\roman*),font=\upshape]
    \item\label{item:rigid1} $\M $ is algebraically rigid.
    \item\label{item:rigid2} $\M $ is $\emptyset$-rigid.
    \item\label{item:rigid3} $\M $ consists of only loops and coloops.
    \item\label{item:rigid4} $\M  \cong \uniform{\ell}{0} * \uniform{c}{c}$ for some $\ell,c \in \N$.
\end{enumerate}
\end{corollary}

\begin{proof}
\ref{item:rigid1}~$\Rightarrow$~\ref{item:rigid2} Follows directly from the definition.\\[1ex]
\ref{item:rigid2}~$\Rightarrow$~\ref{item:rigid3}  The condition \ref{item:rigid2} implies the statement in \Cref{lemma:cotangentloopsandcoloops}~\ref{item:loops3} for any $v \in [n]$.  
Thus, by the same lemma, every element in $[n]$ is either a loop or a coloop.\\[1ex]
\ref{item:rigid3}~$\Rightarrow$~\ref{item:rigid4} The restriction of $\M $ to the loops is isomorphic to $\uniform{\ell}{0}$ and to the coloops to $\uniform{c}{c}$.\\[1ex]
\ref{item:rigid4}~$\Rightarrow$~\ref{item:rigid1}
 As $\uniform{\ell}{0}$ and $\uniform{c}{c}$ are rigid, we conclude by \cite[Proposition 2.3.]{altmann2016rigidity}\footnote{ This proposition says that the join of two complexes is rigid if and only if each of them is rigid.}.
\end{proof}
Matroids satisfying  the equivalent conditions of \Cref{cor:trivialcotangentcohom.} are  called \defi{discrete matroids}. 
So~\Cref{cor:trivialcotangentcohom.} shows that the first  cotangent cohomology does not distinguish among discrete matroids, but that it determines whether  a matroid is discrete or not.
Algebraically rigid simplicial complexes in general are  not classified yet \cite{altmann2016rigidity}.

In the remainder of this section we will show that, for nondiscrete matroids, $T^1$ encodes the entire combinatorial structure. 

\begin{proposition}\label{prop:rkanddimensionsofmatroidsandccm}
  Fix $n\in\N$, let \D\ be a simplicial complex on $[n]$ and $A \subseteq [n]$. We have:
\begin{enumerate}[label=(\roman*),font=\upshape]
    \item\label{item:rank1} 
       $ \rk{}{\Delta} \geqslant~ 1+ \max \Set{\card{A} \sodass \firstcotangentcohomology{}{\link[\Delta]{A}} \neq 0}.
    $
    \item\label{item:rank2}  If \D\ is a matroid with no coloops and $A\!\in\!\D$, then $\lk_\D A$ is rigid if and only if  $A$ is a basis.
    \item\label{item:rank3}  If $\D$ is a matroid that is not discrete, then ~~
    \(
    \rk{}{\D} = 1 + \max \Set{\card{A} \sodass  \firstcotangentcohomology{}{\lk_\D A} \neq 0}.
  \)  
\end{enumerate}
\end{proposition}

\begin{proof}
\ref{item:rank1} Assume that  $\card{A} \geqslant \rk{}{\Delta}$. 
Thus $A$ is either a facet or a nonface, so $\link[\Delta]{A}$ contains only the empty set or is empty, respectively. 
In both cases $\link[\Delta]{A}$ is rigid.
Thus, if $\firstcotangentcohomology{}{\link[\Delta]{A}} \neq 0$, then $\card{A} < \rk{}{\Delta}$.\\[1ex]
\ref{item:rank2} We claim that if a matroid $\D$ has no coloops and $A \in \D$, then $\lk_\D A$ has no coloops.\\[.5ex]
Assume that $v\in\lk_\D A$ is a coloop in the link but not in $\D$. This means, that
\begin{equation}
  \label{eq:coloopBasis}
  \forall B\in\matroid{B}{\D} \text{~if~} A\subseteq B,\text{~then~}v\in B.
\end{equation}
As $v$ is not a coloop of $\D$, there must exist a basis $B'$ with $v\notin B'.$ Choose now one basis $B$ with $A\cup \{v\}\subseteq B$. As $v\in B\sdif B'$, by the basis exchange axiom for matroids we have that there exists $w\in B'\sdif B$ such that
\[
  B'' = (B\sdif v)\cup\{w\}\in\matroid{B}{\D}.
\]
As $A\subseteq B''$ and $v\notin B''$ we obtain a contradiction to (\ref{eq:coloopBasis}). \\[.5ex]   
In summary, by \Cref{cor:trivialcotangentcohom.}, $\lk_\D A$ is rigid if and only if $\lk_\D A = \uniform{[n]\sdif A}{0}$.
Since $\rk{}{\lk_\D A} = \rk{}{\D} - \rk{}{A}$ this can only happen if $A$ is a basis.\\[1ex]
\ref{item:rank3} One inequality follows from~\ref{item:rank1}.
Part~\ref{item:rank2} implies the other inequality if \D\ is coloop free. 
If \D\ has a set of coloops $C$ we can write $\D = \D' * \uniform{\card{C}}{\card{C}}$ with $\D' = \lk_\D C$ being coloop free. 
Again by~\ref{item:rank2} we find a $B \in \D'$ with $\card{B} = \rk{}{\D'}-1$ and $\lk_{\D'}B$ nonrigid. So $B \cup C$ satisfies $\card{(B \cup C)} = \rk{}{\D} - 1$ and by \cite[Corollary~2.4]{altmann2016rigidity} we get
\begin{eqnarray*}
      \firstcotangentcohomology{}{\lk_\D(B \cup C)} &=& \big(\firstcotangentcohomology{}{\lk_{\D'}B} \otimes \K[\uniform{\card{C}}{0}]\big) \oplus \big(\K[\lk_{\D'}B] \otimes \firstcotangentcohomology{}{\uniform{\card{C}}{0}}\big) \\ 
                                                    & =& \big(\firstcotangentcohomology{}{\lk_{\D'}B} \otimes \K[x_1,\dots,x_{\card{C}}] \big) \oplus 0.                                                        
\end{eqnarray*}
The latter is nontrivial because   $\firstcotangentcohomology{}{\lk_{\D'}B} \neq 0$.
\end{proof}

\begin{lemma}\label{prop:reconstructrankonematroids}
If $\M $ is a  matroid of rank one, then we can reconstruct the independent sets of $\M $ from the degreewise dimensions of the first cotangent cohomology module $\firstcotangentcohomology{}{\M }$.
\end{lemma}
\begin{proof}
  For a matroid of rank one every singleton subset of the ground set is either dependent or a basis. Thus, all rank one matroids are of the form
  \[
    \M  = \uniform{m}{1} * \uniform{\ell}{0}.
  \]
Note that any link, with the possible exception of the link of the empty set, has trivial cotangent cohomology.
Thus, the degrees in which $T^1$ is nontrivial have no positive entries.
The matroid is discrete if and only if $m=1$.
If $m = 2$, then the only nontrivial component is $\firstcotangentcohomology{-e_1-e_2}{\M}$, since any element in the ground set is contained in at most one circuit.
Thus, we can recover the independent sets from this one degree.
If $m > 2$, then  $i\in[n]$ is  contained in at least two circuits if and only if $i\in\uniform{m}{1}$.
Therefore $\Set{i}\in\M$ if and only if $\firstcotangentcohomology{-e_i}{\M } \neq 0$.
\end{proof}

We now have all the tools we need to prove the main theorem of this section.

\begin{theorem}\label{thm:ReconstructMatroidsFromCCM}
A matroid $\M $ is discrete if and only if $\firstcotangentcohomology{}{\M}=0.$ If $\firstcotangentcohomology{}{\M}\neq 0$, then we can recover all its independent sets from the dimensions of the graded components of $\firstcotangentcohomology{}{\M }$.
\end{theorem}

\begin{proof}
  The first part is given by~\Cref{cor:trivialcotangentcohom.}. So let us assume that $\M$ is not discrete.
Using \Cref{lemma:cotangentloopsandcoloops} we can recover  all loops and coloops of $\M $. 
We can then write  $\M $ as a join of a loop and coloop free matroid $\M '$ and a discrete matroid $U$.  By \cite[Corollary 2.4]{altmann2016rigidity} we can split the first cotangent cohomology module into a direct sum:
\[
  \firstcotangentcohomology{}{\M } =  \big(\firstcotangentcohomology{}{\M '} \otimes \K[U]\big) \oplus \big(\K[\M '] \otimes \firstcotangentcohomology{}{U}\big).
  \]
Since $\M $ is not discrete we obtain from \Cref{cor:trivialcotangentcohom.}  that $\firstcotangentcohomology{}{\M '} \neq 0$ and $\firstcotangentcohomology{}{U} = 0$.
Therefore, we can use any nontrivial degree of $\M '$ to reconstruct whether an element in the ground set of $U$ is a loop or a coloop.
So it suffices to prove the theorem when $\M$ is a loop-and-coloop-free matroid on $[n]$.
Furthermore, by \Cref{prop:reconstructrankonematroids} we may assume that $\rk{}{\M}>1$.
Let $\bff \in \N^n$ with $F = \supp(\bff)\in\M$ and denote by $\Lambda= \lk_\M F$.
From  $\lk_{\Lambda} G = \lk_\M(F\cup G)$ and 
 (\ref{eq:TiOfLink}) we obtain:
\begin{equation}
  \label{eq:t1oflink}
  \firstcotangentcohomology{}{\Lambda} = 
    \bigoplus_{\substack{\bfc\in\Z^{n}\\ F\,\cap\,\supp\bfc\,=\,\emptyset}} \firstcotangentcohomology{\bfc+\bff}{\M}.  
\end{equation}
Combining (\ref{eq:t1oflink}) with \Cref{lem:L2AC} and \Cref{prop:rkanddimensionsofmatroidsandccm}~\ref{item:rank2}
we have that $\firstcotangentcohomology{}{\Lambda}=0$ if and only if $F$ is a nonface or a basis. We can thus recover all independent sets which are not a basis as:
  \[
    F\in \M\sdif\matroid{B}{\M}    \iff \exists~\bfa,\bfb\in\N^n\text{~with~}F\subseteq \supp\bfa \text{~such that~}\firstcotangentcohomology{\bfa-\bfb}{\M}\neq 0.
  \]
 So we only need to recover the bases. To this aim   it is enough to identify for every $F\in\M$ of rank $\rk{}{\M}-1$ the bases which contain it. This is equivalent to recovering all faces of $\lk_\M F$.   As $\lk_\M F$ has rank one we may apply~\Cref{prop:reconstructrankonematroids} to recover all its faces from $\firstcotangentcohomology{}{\lk_\M F}$. We then conclude by applying (\ref{eq:t1oflink}) again. 
\end{proof}
In summary, the graded first cotangent cohomology module of a matroidal Stanley-Reisner ring can tell us if a matroid is discrete or not. 
If the matroid is discrete, we cannot obtain any more information.
If the matroid is not discrete, we can uniquely determine it.

\bibliographystyle{alpha}
\bibliography{references}

\end{document}